\documentclass{amsart}

\usepackage{amsmath}
\usepackage{amsthm}
\usepackage{amssymb}
\usepackage{amsfonts}
\usepackage{enumitem}
\usepackage{setspace}
\usepackage[bookmarks=true,hyperindex, pdftex,colorlinks,citecolor=blue]{hyperref}
\usepackage{centernot} 
\usepackage{marginnote} 
\usepackage[left=3cm,right=3cm,top=2cm,bottom=2cm]{geometry}
\usepackage{bbm}
\usepackage[all]{xy}
\usepackage{tikz}
\usetikzlibrary{arrows}
\usetikzlibrary{shapes,decorations}
\usetikzlibrary{positioning}

\def\<{\langle}
\def\>{\rangle}

\theoremstyle{plain}
\newtheorem{theorem}{Theorem}[section]
\newtheorem{lemma}[theorem]{Lemma}
\newtheorem{proposition}[theorem]{Proposition}

\theoremstyle{definition}
\newtheorem{definition}[theorem]{Definition}
\newtheorem{example}[theorem]{Example}

\theoremstyle{remark}
\newtheorem{remark}[theorem]{Remark}

\numberwithin{equation}{section}
\usetikzlibrary{matrix}
\usepackage[all]{xy}
\begin{document}

\title[On super-rigid and uniformly super-rigid operators]{On super-rigid and uniformly super-rigid operators}

\author[O. Benchiheb]{Otmane Benchiheb}

\author[F. Sadek]{Fatimaezzahra Sadek}

\author[M. Amouch]{Mohamed Amouch}

\address[Otmane Benchiheb, Fatimaezzahra Sadek and Mohamed Amouch]{University Chouaib Doukkali.
Department of Mathematics, Faculty of science
Eljadida, Morocco} 
\email{otmane.benchiheb@gmail.com}
\email{sadek.fatimaezzahra@yahoo.fr}
\email{amouch.m@ucd.ac.ma}

\date{} 


\begin{abstract}
An operator $T$ acting on a Banach space $X$ is said to be super-recurrent if for each open subset $U$ of $X$, there exist
$\lambda\in\mathbb{K}$ and $n\in \mathbb{N}$ such that $\lambda T^n(U)\cap U\neq\emptyset$.
In this paper, we introduce and study the notions of super-rigidity and uniform super-rigidity which are related to the notion of super-recurrence.
We investigate some properties of these classes of operators and show that they share some properties with super-recurrent operators.
At the end, we study the case of finite-dimensional spaces.
\end{abstract}


\maketitle
\section{Introduction and preliminaries}
Throughout this paper, $X$ will denote a Banach space over the field $\mathbb{K}$ ($\mathbb{K}=\mathbb{R}$ or $\mathbb{C}$).
By an operator, we mean a linear and continuous map acting on $X$,
and we denote by $\mathcal{B}(X)$ the set of all operators acting on $X.$\\
\noindent A (discrete) linear dynamical system is a pair $(X,T)$ consisting of
a Banach space $X$ and an operator $T$.
The most important and studied notions in the linear dynamical system are those of hypercyclicity and supercyclicity.

We say that $T\in\mathcal{B}(X)$ is \textit{hypercyclic} if there exists a vector $x$ whose orbit under $T$; $$Orb(T,x):=\{T^n x \mbox{ : }n\in \mathbb{N}\},$$
 is dense in $X,$ in which case, the vector $x$ is called a \textit{hypercyclic vector} for $T$.
The set of all hypercyclic vectors for $T$ is denoted by $HC(T).$

\noindent An equivalent notion of the hypercyclicity in context of separable Banach spaces, called \textit{topological transitivity}, was introduced by Birkhoff in \cite{Birkhoff}$:$
an operator $T$ acting on a separable Banach space $X$ is hypercyclic if and only if it is topologically transitive; that is, for each pair $(U,V)$ of nonempty open subsets of $X$ there exists $n\in\mathbb{N}$ such that
$$T^n(U)\cap V\neq\emptyset.$$


In $1974$, Hilden and Wallen in \cite{HW} introduced the concept of supercyclicity.
An operator $T$ acting on $X$ is said to be s\textit{upercyclic} if there exists $x\in X$ such that
$$\mathbb{K}Orb(T,x):=\{\lambda T^n x\mbox{ : }\lambda\in\mathbb{K}\mbox{, }n\in\mathbb{N}\}$$
is dense in $X$.
Such a vector $x$ is called a \textit{supercyclic vecto}r for $T$.
The set of all supercyclic vectors for $T$ is denoted by $SC(T).$\\
\noindent An operator $T$ acting on a separable Banach space $X$ is supercyclic if and only if for each pair
$(U,V)$ of nonempty open subsets of $X$ there exist $\lambda\in\mathbb{K}$ and $n\in\mathbb{N}$ such that
$$\lambda T^n(U)\cap V\neq\emptyset,$$
see \cite[Theorem 1.12]{BM}.

For more information about hypercyclic and supercyclic operators and their proprieties, see K.G. Grosse-Erdmann and A. Peris's book \cite{Peris} and F. Bayart and E. Matheron's book \cite{BM}, and the survey article \cite{Grosse} by K.G. Grosse-Erdmann.

Another important notion in the linear dynamical system that has a long story is that of recurrence
 which introduced by Poincar\'{e} in \cite{Poincare}.
Later, it have been studied by Gottschalk and Hedlund \cite{GH} and also by Furstenberg \cite{Furstenberg}.
Recently, recurrent operators have been studied in \cite{CMP}.

We say that $T\in\mathcal{B}(X)$ is \textit{recurrent} if for each open subset $U$ of $X$,
there exists $n\in\mathbb{N}$ such that
$$T^n(U)\cap U\neq\emptyset.$$
A vector $x\in X$ is called a recurrent vector for $T$ if there exists an increasing sequence $(n_k)\subset\mathbb{N}$ such that
$$T^{n_k}x\longrightarrow x \ \ \mbox{ as } \ \ k\longrightarrow\infty.$$
We denoted by $Rec(T)$ the set of all recurrent vectors for $T$.
and we have that $T$ is recurrent if and only if $Rec(T)$ is dense in $X.$
For more information about this classe of operators, see \cite{Ak,BGLP,GM,CP,E,EG,GMOP,GMQ,HZ,YW}.

\noindent Recall that an operator $T\in\mathcal{B}(X)$ is said to be \textit{rigid} if there exists a strictly increasing sequence $(n_k)\subset\mathbb{N}$ such that
$$T^{n_k}x\longrightarrow x, \ \ \mbox{ as } \ \ k\longrightarrow \infty, \ \ 
\mbox{ for all } \ \ x\in X.$$
This means that each vector of the space $X$ is a recurrent vector for $T$ with respect to the same sequence $(n_k).$\\
\noindent An operator $T$ is said to be \textit{uniformly rigid} if there exists a strictly increasing sequence $(n_k)\subset \mathbb{N}$ such that
$$\Vert T^{n_k}-I\Vert=\sup_{\Vert x\Vert\leq1}\Vert T^{n_k}x-x\Vert\longrightarrow 0,\ \ \mbox{ as } k\longrightarrow\infty.$$
Clearly, we have that
$$  T\mbox{ is uniformly rigid } \Rightarrow T\mbox{ is rigid }\Rightarrow T\mbox{ is recurrent}. $$
The converses of those implications does not hold in general, see \cite{CMP}.

Recently in \cite{ABS}, recurrent operators have been generalized to a large class of operators called \textit{super-recurrent} operators.
We say that $T\in\mathcal{B}(X)$ is super-recurrent if for each open subset $U$ of $X$, there exist $\lambda\in\mathbb{K}$ and $n\in\mathbb{N}$ such that
$$\lambda T^{n}(U)\cap U\neq\emptyset.$$
A vector $x\in X$ is called a super-recurrent vector for $T$ if there exist an increasing sequence $(n_k)\subset\mathbb{N}$ and a sequence $(\lambda_k)\subset \mathbb{K}$ such that
$$\lambda_k T^{n_k}x\longrightarrow x\ \mbox{as}\ k\longrightarrow\infty.$$
We denoted by $SRec(T)$ the set of all recurrent vectors for $T$.

Taking into consideration the link between recurrence and their deviation,
 we introduce in this paper the notions of super-rigidity and uniform super-rigidity which are related to the notion of super-recurrence and generalize the notions of rigidity and uniform rigidity.

In Section $2$, we study the notion of super-rigidity.
We give the relationship between super-rigid, rigid, super-recurrent, and recurrent operators.
Also we prove that the super-rigidity is preserved under similarity and we give the relationship between the super-rigidity of an operator $T$ and its iterates by showing that $T$ is super-rigid if and only for each $p\geq2$, the operator $T^p$ is super-rigid.
At the end of the section, several spectral properties of super-rigid operators will be proven.

In Section $3$, we introduce and study the notion of uniformly super-rigid operators.
As in the super-rigidity's case, we prove that the property of being uniform super-rigid is preserved under similarity and that $T$ is uniformly super-rigid if and only if $T^p$ is uniformly super rigid for all $p.$
Moreover, we prove that the spectrum of a uniform-super-rigid operator has a specific property.
This lead us to show that being invertible is a necessary condition of being uniformly super-rigid.

In Section $4$, we characterize the super-recurrence in finite-dimensional spaces, 
and we show that in this case, the notions of super-recurrent, super-rigid and uniformly super-rigid are equivalent.

\section{Super-rigid operators}
In the following, we introduce the notion of \textit{super-rigid} operators which generalizes the notion of rigidity and related to the notion of super-recurrence. 
\begin{definition}
An operator $T$ acting on $X$ is called super-rigid if there exist a strictly increasing sequence of positive integers $(n_k)_{n\in\mathbb{N}}$ and a sequence $(\lambda_{k})_{n\in\mathbb{N}}$ of numbers such that
$$ \lambda_{k} T^{n_k}x\longrightarrow x$$
for every $x\in X.$
\end{definition}

\begin{example}\label{exa}
For $1< p<\infty$, let $X=\ell^p(\mathbb{N})$.
Let $R$ be a strictly positive number and $(\lambda_k)$ a sequence of numbers such that
$\lambda_k\in\{z \in\mathbb{C} \mbox{ : } \vert z \vert =R \}$ for all $k$.
Let $T$ be an operator defined on $\ell^p(\mathbb{N})$ by
$$  T(x_1,x_2,\dots)=(\lambda_1 x_1,\lambda_2 x_2,\dots), \hspace{0.3cm} \mbox{ for all } (x_k)\in \ell^p(\mathbb{N}).$$
Then $T$ is super-rigid. Indeed, this is since $R^{-1}T $ is recurrent, see \cite[Theorem 5.4]{CMP}.
\end{example}
\begin{remark}\label{remp}
It is clear that the notion of super-rigidity implies the notion of rigidity,
but the converse does not hold in general.
Indeed, let $T$ be the operator defined in Example \ref{exa}, then $T$ is super-rigid.
However, $T$ is recurrent if and only if $\lambda_k=1$, for all $k,$
see \cite[Theorem 5.4]{CMP}.
\end{remark}
Let $T$ be an operator acting on $X$. It is clear that
$$ T \mbox{ is super-rigid } \Rightarrow T \mbox{ is super-recurrent. } $$
However, the converse does not hold in general even if $T$ is supercyclic as we show in the next example.
\begin{example}\label{exs}
Let $(e_n)_{n\in\mathbb{N}}$ be the canonical basis of $\ell^2(\mathbb{N})$ and $\textbf{w}=(\omega_n)_{n\in\mathbb{N}}$ be a bounded sequence of positive numbers.
Let $B_\textbf{w}$ be the weighted backward shift operator, defined on $\ell^2(\mathbb{N})$ by
$$ B_\textbf{w}(e_0)=0 \hspace{0.2cm} \mbox{ and } \hspace{0.2cm} B_\textbf{w}(e_n)=\omega_n e_{n-1} \hspace{0.1cm} \mbox{ for all }n\geq1. $$
$B_\textbf{w}$ is super-recurrent since it is supercyclic, see \cite[Example 1.15]{BM}.
However, $B_\textbf{w}$ cannot be super-rigid since $B_\textbf{w}(e_1)=0$.
\end{example}
We have the following diagram showing the relationships among super-rigidity, rigidity, super-recurrence, and recurrence. 
\begin{center}
\begin{tikzpicture}
  \matrix (m) [matrix of math nodes,row sep=2em,column sep=5em,minimum width=5em]
  {
    \mbox{ Rigidity }&\mbox{ Recurrence }\\
    \mbox{ Super-rigidity} & \mbox{Super-recurrence} \\};
  \path[-stealth]
    (m-1-1) edge  node [left] {$\not\uparrow_{\mbox{ \tiny{Remark \ref{remp} }}}$} (m-2-1)
            edge  node [below] {$\nleftarrow_{\mbox{\tiny{\cite[Example 5.4]{CMP}}}}$} (m-1-2)
    (m-2-1.east|-m-2-2) edge  node [below] {}
            node [above] {$\nleftarrow_{\mbox{\tiny{Example \ref{exs}}}}$} (m-2-2)
    (m-1-2) edge  node [right] {$\not\uparrow_{\mbox{\tiny{\cite[Remarks 2.2]{ABS}}}}$} (m-2-2);
\end{tikzpicture}
\end{center}
In the following proposition, we prove that the super-rigidity is preserved under similarity.
\begin{proposition}
Let $T\in\mathcal{B}(X)$ and $S\in \mathcal{B}(Y)$.
Assume that $T$ and $S$ are similar.
Then $T$ is super-rigid on $X$ if and only if $S$ is super-rigid on $Y$.
\end{proposition}
\begin{proof}
Since $T$ and $S$ are similar,
then there exists a homeomorphisme $\phi$ : $X\longrightarrow Y$ such that $S\circ\phi=\phi\circ T.$
Let $y\in Y$, then there exists $x\in X$ such that $y=\phi x.$
Since $T$ is super-rigid in $X$, there exist a strictly increasing sequence $(n_k)$ of positive integers and a sequence $(\lambda_{k})$ of numbers such that
$ \lambda_{k}T^{n_k}x\longrightarrow x $ as $k\longrightarrow+\infty$.
Since $\phi$ is continuous and $S\circ\phi=\phi\circ T$, it follows that
$ \lambda_{k}S^{n_k}y\longrightarrow y $
as $k\longrightarrow+\infty$,
which means that $S$ is super-rigid.
\end{proof}
Let $p$ be a nonzero fixed positive integer.
In $1995$, Ansari proved that $T$ is a hypercyclic (resp, supercyclic) operator on a separable Banach space if and only if $T^p$ is hypercyclic (resp, supercyclic), see \cite{Ansari}.
Later, the same result was proven for recurrent operators, see \cite[Proposition 2.3]{CMP}, and for super-recurrent operators, see \cite[Theorem 3.11]{ABS}.
In the following theorem, we prove that this result remains true for super-rigid operators. 
\begin{theorem}
Let $T$ be an operator acting on $X$. Then $T$ is super-rigid if and only if $T^p$ is super-rigid for all $p\geq2$.
\end{theorem}
\begin{proof}
Assume that $T$ is super-rigid, then there exist a strictly increasing sequence $(n_k)$ of positive integers and a sequence $(\lambda_{k})$ of numbers such that
$\lambda_{k} T^{n_k}x\longrightarrow x$ for every $x\in X$.
Let $M:=\sup_{n\in\mathbb{N}}\Vert\lambda_{k}T^{n_k}\Vert$. By Banach-Steinhaus theorem we have $M<+\infty.$
Let $x$ be a vector of $X$, then we have
\begin{align*}
\Vert \lambda_{k}^{p} T^{pn_k}x-x \Vert&=\Vert \lambda_{k}^{p}T^{pn_k}x-\lambda_{k}^{p-1}T^{(p-1)n_k}x+\lambda_{k}^{p-1}T^{(p-1)n_k}x-\dots+\lambda_{k}T^{n_k}x-x \Vert\\
                                       &\leq \Vert \lambda_{k}^{p}T^{pn_k}x-\lambda_{k}^{p-1}T^{(p-1)n_k}x \Vert+\dots+\Vert \lambda_{k}T^{n_k}x-x \Vert\\
                                        &\leq \Vert\lambda_{k}^{p-1}T^{p-1}\Vert\Vert \lambda_{k}T^{n_k}x-x \Vert+\dots+\Vert \lambda_{k}T^{n_k}x-x \Vert\\
                                        &\leq \left( \sum_{i=0}^{p-1}M^i \right)\Vert \lambda_{k}T^{n_k}x-x \Vert.
\end{align*}
Since $T$ is super-rigid with respect to $(n_k)$ and $(\lambda_{k})$, it follows that 
$\lambda_{k}^{p} T^{pn_k}x\longrightarrow x,$
which means that $T^p$ is super-rigid with respect to  $(n_k)$ and $(\lambda_{k}^p)$.\\
\end{proof}
In the following, we will give a characterization of the spectrum of super-rigid operator.
We begin with the following lemma.
\begin{lemma}\label{lem1}
Let $T$ be an operator acting on $X$ and $\lambda$ a nonzero number.
Then, $T$ is super-rigid if and only if $\lambda T$ is super-rigid.
\end{lemma}
\begin{proof}
Assume that $T$ is super-rigid.
then there exists a strictly increasing sequence $(n_k)\subset \mathbb{N}$ and $(\lambda_k)\subset\mathbb{K}$ such that $\lambda_k T^{n_k}x\longrightarrow x$ for every $x\in X$.
Let $\mu_k=\lambda^{-n_k}\lambda_k$.
Then
$$\mu_k (\lambda T)^{n_k}x=\lambda^{-n_k}\lambda^{n_k}\lambda_k T^{n_k}x=\lambda_k T^{n_k}x\longrightarrow x,$$
for every $x\in X$.
This implies that $\lambda T$ is super-rigid.
\end{proof}

The next proposition shows that the eigenvalues  of a super-rigid operators has the same argument.
\begin{proposition}\label{pr1}
Let $T$ be an operator acting on $X$.
If $T$ is super-rigid, then
$$\sigma_p(T)\subset \{z\in\mathbb{C} \mbox{ : }\vert z \vert=R\},$$
for some strictly positive real number $R$.
\end{proposition}
\begin{proof}
Assume that there exist $\lambda_1$ and $\lambda_2$ in the point spectrum of $T$ such that
$\vert\lambda_1\vert<\vert\lambda_2\vert$.
Let $m$ be a strictly positive real number such that $\vert\lambda_1\vert<m<\vert\lambda_2\vert$.
Since $\lambda_1$, $\lambda_2\in \sigma_p(T)$, it follows that there exist $x$, $y\in X\setminus\{0\}$
such that $Tx=\lambda_1 x$ and $Ty=\lambda_2 y.$
By Lemma \ref{lem1}, the operator $\frac{1}{m}T$ is super-rigid.
Hence, there exist a sequence $(\mu_k)\subset \mathbb{K}$ and a sequence $(n_k)\subset\mathbb{N}$ such that
\begin{equation}\label{eo1}
\mu_k\left( \frac{\lambda_1}{m}\right)^{n_k}x\longrightarrow x  \hspace{0.2cm}
\end{equation}
 and 
 \begin{equation}\label{eo2}
\hspace{0.2cm} \mu_k\left( \frac{\lambda_2}{m}\right)^{n_k}y\longrightarrow y. 
 \end{equation}
By (\ref{eo1}), we have $\vert \mu_k \vert\longrightarrow +\infty$, and by (\ref{eo2}), we have $\vert \mu_k\vert\longrightarrow 0$, which is a contradiction.
\end{proof}
\begin{remark}
Proposition \ref{pr1} is not true in general for super-recurrent operators which are not super-rigid.
Indeed, let $(e_n)_{n\in\mathbb{N}}$ be the canonical basis of $\ell^2(\mathbb{N})$ and $\textbf{w}=(\omega_n)_{n\in\mathbb{N}}$ be a bounded sequence of positive numbers.
Suppose that $B_\textbf{w}$ is defined as in Example \ref{exs}, then $B_\textbf{w}$ is super-recurrent.
However, if $\omega_n=1$, for all $n\geq1$, and $B_\textbf{w}=B$ is the backward shift operator, then
$$ \sigma_p(B)=B(0,1). $$
\end{remark}

Let $T$ be a super-rigid operator acting on a Banach space $X$.
Since any super-rigid operator is super-recurrent, it follows by \cite[Theorem 4.2]{ABS} that
that if $T$ is super-rigid, then the eigenvalues of $T^{*}$; the Banach adjoint operator of $T$, are of the same argument.
This means that there exists $R_1>$ such that $\sigma_p(T^*)\subset \{z\in\mathbb{C} \mbox{ : }\vert z \vert=R_1\}$.
Moreover, by Proposition \ref{pr1}, 
the eigenvalues of $T$ are of the same argument.
This means that
there exists $R_2>$ such that $\sigma_p(T)\subset \{z\in\mathbb{C} \mbox{ : }\vert z \vert=R_2\}$.
The next proposition shows that we have $R_1=R_2.$ This means that the eigenvalues of $T$ and the eigenvalues of his Banach adjoint are of the same argument.
\begin{proposition}\label{pr2}
Let $T$ be an operator acting on $X$.
If $T$ is super-rigid, then the eigenvalues of $T$ and the eigenvalues of his Banach adjoint are of the same argument.
This means that there exists some $R>0$ such that
$$\sigma_p(T) \cup \sigma_p(T^{*})\subset \{z\in\mathbb{C} \mbox{ : }\vert z \vert=R\}.$$
\end{proposition}
\begin{proof}
Assume that there exist $\lambda_1\in \sigma_p(T)$ and $\lambda_1\in \sigma_p(T^{*})$ such that
$ \vert\lambda_1\vert<m<\vert\lambda_2\vert, $
where $m$ is a strictly positive real number.
Since $\lambda_1\in \sigma_p(T)$ and $\lambda_1\in \sigma_p(T^{*})$,
it follows that there exist $x\in X\setminus\{0\}$ and $x^{*}\in X^{*}\setminus\{0\}$ such that
$Tx=\lambda_1 x$ and $T^{*}x^{*}=\lambda_2 x^{*}$.
By lemma \ref{lem1}, the operator $\frac{1}{m}T$ is super-rigid.
Hence, there exist a sequence $(\mu_k)\subset \mathbb{K}$ and a sequence $(n_k)\subset\mathbb{N}$ such that
$\mu_k\left( \frac{1}{m}T \right)^{n_k}y\longrightarrow y$, for all $y\in X.$
In particular, for $y=x,$ we have
\begin{equation}\label{eq1}
\mu_k\left( \frac{1}{m}T \right)^{n_k}x=\mu_k\left( \frac{\lambda_1}{m} \right)^{n_k}x\longrightarrow x.
\end{equation}
Since $x^{*}$ is an nonzero linear form on $X$, it follows that there exists $z\in X$ such that $x^{*}(z)\neq0.$
Since $\frac{1}{m}T$ is super-rigid and $x^{*}$ is continuous, we have
\begin{equation}\label{eq2}
\mu_k\left( \frac{\lambda_2}{m} \right)^{n_k}x^{*}(z)\longrightarrow x^{*}(z).
\end{equation}
\end{proof}
By (\ref{eq1}), we have $\vert \mu_k\vert\longrightarrow+\infty$ and by (\ref{eq2}), we have $\vert \mu_k\vert\longrightarrow0$,
which is a contradiction.
\begin{example}
Let $T$ be the operator defined in Example \ref{exa}.
If $p=2,$ then it is easy to prove that
$ \sigma_p(T)=\{\lambda_1,\lambda_2,\dots\}. $
Moreover, a simple verification shows that the adjont of $T$ is defined by
$$ T^{*}(x_1,x_2,\dots)=(\overline{\lambda_1}x_1, \overline{\lambda_2}x_2,\dots). $$
Hence,
$ \sigma_p(T^*)=\{\overline{\lambda_1},\overline{\lambda_2},\dots\}. $
This shows that
$$\sigma_p(T) \cup \sigma_p(T^{*})\subset \{z\in\mathbb{C} \mbox{ : }\vert z \vert=R\}.$$
\end{example}
\section{Uniformly super-rigid operators}
In the following, we introduce the notion of uniform super-rigidity which generalizes the notion of uniform rigidity.
\begin{definition}
An operator $T$ acting on $X$ is called uniformly super-rigid if there exist a strictly increasing sequence of positive integers $(n_k)_{n\in\mathbb{N}}$ and a sequence $(\lambda_{k})_{n\in\mathbb{N}}$ of  numbers such that
$$ \Vert \lambda_{k} T^{n_k}-I\Vert=\sup_{x\neq0}\Vert\lambda_{k} T^{n_k}x-x \Vert\longrightarrow0.$$
\end{definition}
\begin{example}\label{exab}
Let $T$ be an operator defined on $\mathbb{C}^n$ by$:$
$
T (x_1,\dots,x_n) =(\lambda_1 x_1,\dots,\lambda_n x_n)$,
where $\lambda_i\in\mathbb{K}$ and $\lambda_i=\lambda_j=R>0$, for $1\leq i, j\leq n.$
Let $(x_1,\dots x_n)\neq(0,\dots 0)$ and $m\geq 0$, we have
$$ R^{-m}T^m (x_1,\dots x_n)-(x_1,\dots x_n)=(R^{-m}\lambda_1^m x_1,\dots R^{-m}\lambda_n^m x_n). $$
Since 
$\lambda_i=\lambda_j=R>0$, for $1\leq i, j\leq n,$ it follows that there exists a strictly increasing sequence $(m_k)$ of positive integers such that
$$ \Vert R^{m_k} T^{m_k}(x_1,\dots x_n)-(x_1,\dots x_n) \Vert\longrightarrow 0. $$
This means that $T$ is a uniformly super-rigid operator.
\end{example}
\begin{remark}
It is clear that the uniform super-rigidity implies the uniform rigidity.
However, the converse does not hold in general.
Indeed,
let $T$ be the operator defined as in the Example \ref{exab}, then $T$ is a uniformly super-rigid whenever 
$\lambda_i=\lambda_j=R>0$, for $1\leq i, j\leq n.$
But the operator $T$ is uniform rigid if and only if $\lambda_i=\lambda_j=1$, for $1\leq i, j\leq n,$ see \cite[Section 4]{CMP}.
\end{remark}
\begin{lemma}\label{lem2}
Let $T$ be an operator acting on $X$.
If $T$ is uniformly super-rigid, then $\lambda T$ is uniformly super-rigid for all nonzero number $\lambda.$
\end{lemma}
\begin{proof}
If $T$ is uniformly super-rigid with respect to a sequence $(\mu_k)$ and a sequence $(n_k)$,
then it is easy to show that $\lambda T$ is uniformly super-rigid with respect to $(\mu_k \lambda^{-n_k})$ and $(n_k).$
\end{proof}
\begin{remark}
It is clear that each uniform super-rigid operator is super-rigid.
However, the converse does not hold in general.
Indeed,
let $(\theta_k)_{k\in\mathbb{N}}$ be a sequence of real numbers such that
$$ \liminf_{n\longrightarrow+\infty}\left(  \sup_{k\in\mathbb{N}} \lvert e^{2\pi in \theta_k}-1\rvert \right)\nrightarrow0.  $$
Let $R>0$, and $\lambda_k= R e^{2\pi in \theta_k}$, for all $k\in\mathbb{N}$.
Let $T$ be the operator defined as in Example \ref{exa}.
Then $T$ is super-rigid.
On the other hand, assume that if $T$ is uniformly super-rigid, then $R^{-1}T$ is uniformly rigid,
which is not possible by \cite[Theorem 5.4]{CMP}.
Hence, $T$ is super-rigid but not uniformly super-rigid.
\end{remark}
In the following proposition, we prove that the uniform super-rigidity is preserved under similarity.
\begin{proposition}
Let $T\in\mathcal{B}(X)$ and $S\in \mathcal{B}(Y)$.
Assume that $T$ and $S$ are similar.
Then $T$ is uniformly super-rigid on $X$ if and only if $S$ is uniformly super-rigid on $Y$.
\end{proposition}
\begin{proof}
Since $T$ and $S$ are similar,
then there exists a homeomorphisme $\phi$ : $X\longrightarrow Y$ such that $S\circ\phi=\phi\circ T.$
\\
\noindent Since $T$ is uniformly super-rigid, it follows that a strictly increasing sequence of positive integers $(n_k)_{n\in\mathbb{N}}$ and a sequence $(\lambda_{k})_{n\in\mathbb{N}}$ of numbers such that
$$ \Vert \lambda_{k} T^{n_k}-I\Vert=\sup_{x\neq0}\Vert\lambda_{k} T^{n_k}x-x \Vert\longrightarrow0.$$

\noindent Let $y$ be a nonzero vector of $Y$ and pick $x$ a nonzero vector of $X$ such that $y=\phi(x)$.
Then
\begin{align*}
\Vert \lambda_{k} S^{n_k}-I\Vert&=\sup_{y\neq0}\Vert\lambda_{k} S^{n_k}y-y \Vert\\
                                &=\sup_{x\neq0}\Vert\lambda_{k} S^{n_k}\circ\phi( x)-\phi(x) \Vert\\
                                &=\sup_{x\neq0}\Vert \phi(\lambda_{k} T^{n_k}x-x) \Vert\\
                                &\leq \Vert\phi\Vert \sup_{x\neq0}\Vert\lambda_{k} T^{n_k}x-x \Vert\longrightarrow0.
\end{align*}
Hence $S$ is uniformly super-rigid on $Y$.
\end{proof}
In the following theorem, we give the relationship between the uniform super-rigidity of an operator and its iterations.
\begin{theorem}\label{thm1}
Let $T$ be an operator acting on $X$. Then
$T$ is uniformly super-rigid if and only if $T^p$ is uniformly super-rigid, for all $p\geq2$.
\end{theorem}
\begin{proof}
 Let $p$ be a strictly positive integer.
Assume that $T$ is uniformly super-rigid, then there exist a strictly increasing sequence of positive integers $(n_k)_{k\in\mathbb{N}}$ and a sequence $(\lambda_{k})_{k\in\mathbb{N}}$ of numbers such that
$ \Vert \lambda_{k} T^{n_k}-I\Vert\longrightarrow0.$
By Banach-Steinhaus theorem we have $\sup_{n\in\mathbb{N}}\Vert\lambda_{k}T^{n_k}\Vert=M:=<+\infty.$
It follows that
\begin{align*}
\Vert \lambda_{k}^{p} T^{pn_k}-I \Vert&\leq \Vert \lambda_{k}^{p-1}T^{(p-1)n_k}+\lambda_{k}^{p-2}T^{(p-2)n_k}+ \dots+I \Vert\Vert \lambda_{k}T^{n_k}-I \Vert\\
                                        &\leq \left( \sum_{i=0}^{p-1}M^i \right)\Vert \lambda_{k}T^{n_k}-I \Vert.
\end{align*}
This shows that $T^p$ is uniformly super-rigid whenever $T$ is uniformly super-rigid.
\end{proof}
Let $T$ be an operator acting on $X$.
If $T$ is super-rigid,
then by Proposition \ref{pr2}, there exists some $R>0$ such that 
$\sigma_p(T) \cup \sigma_p(T^{*})\subset \{z\in\mathbb{C} \mbox{ : }\vert z \vert=R\}.$
For uniformly super-rigid operators we have a significant strengthening of this result.
\begin{theorem}\label{thp}
Let $T$ be an operator acting on $X$.
If $T$ is uniformly super-rigid, then there exists $R>$ such that
$$\sigma(T)\subset \{z\in\mathbb{C} \mbox{ : }\vert z \vert=R\}.$$
\end{theorem}
\begin{proof}
Assume that $T$ is uniformly super-rigid.
Then there exist a sequence $(\mu_k)\subset\mathbb{K}$ and a sequence $(n_k)\subset \mathbb{N}$ such that
$ \Vert \mu_k T^{n_k}-I\Vert\longrightarrow0 $ as $k\longrightarrow\infty.$

\noindent Without loss of generality, by Lemma \ref{lem2}, we may suppose that $\Vert T \Vert=1.$
By Proposition \ref{pr2}, we have $\sigma_p(T) \cup \sigma_p(T^{*})\subset \{z\in\mathbb{C} \mbox{ : }\vert z \vert=R\}$ for some $0< R \leq1$.
Since $\sigma_p(T) \cap \mathbb{D}=\emptyset$, it follows that $R=1.$

\noindent If $\lambda\in \sigma(T)\cap\mathbb{D}$, then by the previous discussion, $\lambda$ is necessarily in $\sigma_a(T)$, the approximate point spectrum of $T.$

\noindent By contradiction, assume that there exists $\lambda\in\mathbb{K}$ such that $\lambda\in\sigma(T)\cap\mathbb{D}$.
Since $\lambda\in \sigma(T)$, it follows by the spectral theorem that $\lambda^p\in \sigma(T^p),$ for all $p\ge2.$
On the other hand, by Theorem \ref{thm1}, the operator $T^p$ is uniformly super rigid for all $p\ge2.$
Hence $\lambda^p$ is in the approximate point spectrum of $T^p$.
Thus, for all $p\ge2,$ there exists a sequence $(x_k^{(p)})_{k\in\mathbb{N}}\subset \{z\in \mathbb{C}\mbox{ : }\vert z\vert=1\}$
such that $\Vert Tx_k^{(p)})-\lambda x_k^{(p)})  \Vert \longrightarrow0$ as $k\longrightarrow\infty.$
Using this, one can find sequence $(y_k)_{k\in\mathbb{N}}$ of $X$ such that $(y_k)_{k\in\mathbb{N}}\subset \{\vert z\vert=1\}$ and $\Vert T^k y_k-\lambda y_k \Vert<a_k$,
where $(a_k)$ is such that $(\vert\mu_k\vert(a_k+\vert\lambda\vert^k)\longrightarrow0$ as $k\longrightarrow\infty$.
This implies that 
$$ \Vert \mu_k T^ky_k \Vert\leq 
\vert \mu_k\vert (a_k+\vert\lambda\vert^k)
\longrightarrow0 \mbox{ as }k\longrightarrow\infty.$$
On the other hand
$$ \left\lvert \mbox{ }\Vert \mu_k T^{n_k}y_{n_k}\Vert-1 \right\rvert\leq \Vert \mu_k T^{n_k}y_{n_k} - y_{n_k} \Vert\leq \Vert \mu_k T^{n_k}-I \Vert  \longrightarrow0 \mbox{ as }k\longrightarrow\infty,$$
which is a contradiction.
\end{proof}
\begin{remark}
The inclusion in Theorem \ref{thp} could be a strict inclusion.
Indeed, let $\lambda$ be an nonzero number, then $\lambda I$ is a uniformly super-rigid operator.
On the other hand, let $\vert \lambda\vert=R$.
Then
$$\sigma(T)=\{\lambda\}\varsubsetneq \{z\in\mathbb{C} \mbox{ : }\vert z \vert=R\}.$$
\end{remark}
\begin{proposition}
Let $T$ be an operator acting on $X$.
If $T$ is uniformly super-rigid, then it is invertible.
\end{proposition}
\begin{proof}
Assume that $T$ is uniformly super-rigid.
By Theorem \ref{thp}, there exists $R>$ such that
$\sigma(T)\subset \{z\in\mathbb{C} \mbox{ : }\vert z \vert=R\}$.
Hence, $0\notin \sigma(T)$, which means that $T$ is invertible.
\end{proof}
\section{Finite Dimensional Spaces}
In this section, we will characterize the super-recurrence, super-rigidity, and uniform super-rigidity in finite-dimensional space.

In the following, if $T\in \mathcal{B}(\mathbb{K}^d)$, then we denote by $A$ a matrix of $T$.
Moreover, if $T$ is super-recurrent, let $R>0$ be such that each component of the spectrum of $T$ intersects the circle $\{z\in\mathbb{C} \mbox{ : }\vert z \vert=R\}$, see \cite[Theorem 4.1]{ABS}.

In the complex case, we have then the following theorem.
\begin{theorem}
Let $T\in\mathcal{B}(\mathbb{C}^d)$.
Then the following statements are equivalent$:$
\begin{enumerate}
\item $T$ is super-recurrent;
\item $T$ is super-rigid;
\item $T$ is uniformly super-rigid;
\item $A$ is similar to a diagonal matrix with entries of the same argument.
\end{enumerate}
\end{theorem}
\begin{proof}
We need only to prove that $(4)\Rightarrow (3)$ and $(1)\Rightarrow (4).$\\

\noindent $(4)\Rightarrow (3)$ : Assume that $A$ is similar to a diagonal matrix with entries $\lambda_1,\dots,\lambda_d\in\{z\in\mathbb{C} \mbox{ : }\vert z \vert=R\}$ for some $R>0$.
Then there exists an increasing sequence $(n_k)$ such that $(R^{-1}\lambda_i)^{n_k}\longrightarrow 1$, for $1\leq i\leq d$.
This implies that that $T$ is uniformly super-rigid.

\noindent $(1)\Rightarrow (4)$ : Assume that $T$ is super-recurrent.
Using the fact that $\sigma_p(T)=\sigma(T)$ and \cite[Theorem 4.1]{ABS}, we conclude that 
$\sigma(T)=\{\lambda_1,\dots,\lambda_M\}\subset\{z\in\mathbb{C} \mbox{ : }\vert z \vert=R\},$ where each $\lambda_i$ with multiplicities $m_i.$
By Jordan decomposition theorem, the matrix $A$ is similar to a matrix of the form
$$\begin{pmatrix}
A_1 & 0 & \dots & 0 \\ 
0 & A_2 & \dots & 0 \\ 
\vdots & \vdots & \ddots & \vdots \\ 
0 & 0 & \dots & A_M
\end{pmatrix},  $$
where each $A_j$ has the form $A_j=\lambda_j I_{m_j}$ or the form
\begin{equation}\label{eq13}
A_j=\begin{pmatrix}
\lambda_j & 1 & 0 & \dots & 0 \\ 
0 & \lambda_j & 1 & \dots & 0 \\ 
\vdots & \vdots & \ddots & \ddots & \vdots \\ 
0 & 0 & \dots & \lambda_j & 1 \\ 
0 & 0 & \dots & 0 & \lambda_j
\end{pmatrix}.   
\end{equation}
Assume that there exists a block $A_{j_0}$ has the form  $(\ref{eq13})$, then $m_j\geq2$.
Using \cite[Proposition 3.7]{ABS}, we conclude that the operator presented by the matrix
$$
B=\begin{pmatrix}
\lambda_{j_0} & 1 \\ 
0 & \lambda_{j_0}
\end{pmatrix} 
$$
is super-recurrent on $\mathbb{C}^2$.
By straightforward induction, we have for all $n\in\mathbb{N}$
$$ 
B^n=\begin{pmatrix}
\lambda_{j_0}^n & n  \lambda_{j_0}^{n-1}\\ 
0 & \lambda_{j_0}^n
\end{pmatrix}.
$$
Let $(z_1,z_2)\in\mathbb{C}^2$ be a super-recurrent vector for $B$ with $z_2\neq0$.
Then there exist $(\mu_k)\subset \mathbb{C}$ and a strictly increasing sequence of positive integers $(n_k)$ such that
$$
\mu_k
\begin{pmatrix}
\lambda_{j_0}^{n_k} & n  \lambda_{j_0}^{n_k-1}\\ 
0 & \lambda_{j_0}^{n_k}
\end{pmatrix}
\begin{pmatrix}
z_1 \\ 
z_2
\end{pmatrix} 
\longrightarrow
(z_1,z_2).
$$
which implies that
$$ \mu_k \lambda_{j_0}^{n_k} z_2\longrightarrow z_2 \hspace{0.2cm}\mbox{ and }\hspace{0.2cm}
\mu_k \lambda_{j_0}^{n_k}z_1+\mu_k n_k\lambda_{j_0-1}^{n_k}z_2\longrightarrow z_1.$$
This is impossible since by hypothesis $z_2\neq0.$
\end{proof}
In the real case, we have the following theorem.
\begin{theorem}
Let $T\in \mathbb{R}^d$.
The following assertions are equivalent$:$
\begin{enumerate}
\item $T$ is super-recurrent;
\item $T$ is super-rigid;
\item $T$ is uniformly super-rigid;
\item $A$ is similar to a matrix of the form
\begin{equation}\label{eq24}
\begin{pmatrix}
A_1 & 0 & \dots & 0 \\ 
0 & A_2 & \dots & 0 \\ 
\vdots & \vdots & \ddots & \vdots \\ 
0 & 0 & \dots & A_M
\end{pmatrix}, 
\end{equation}
where each $A_k$, $1\leq k\leq M$, is $1\times 1$ matrix with entry $R$ or $-R$, or $2\times 2$ matrix of the form
\begin{equation}\label{eq34}
\begin{pmatrix}
a & b \\ 
-b & a
\end{pmatrix},
\hspace{0.3cm} a\mbox{, }b\in\mathbb{R}.
\end{equation}
\end{enumerate}
\end{theorem}
\begin{proof}
We need only to prove that $(4)\Rightarrow (3)$ and $(1)\Rightarrow (4).$\\

\noindent $(1)\Rightarrow (4)$ :
Assume that $T$ is super-recurrent.
By the Jordan decomposition, the matrix $A$ is similar to a matrix which has the form of $(\ref{eq24}).$
We have then two cases.

$\textbf{(1)}$ : Each $A_k$ has the form $\lambda_k I_{m_k}$.
In this case, it suffices to do the same method which was used in the complex case to conclude.

$\textbf{(2)}$ : Each $A_k$ has the form
\begin{equation}\label{eq44}
\begin{pmatrix}
B & I_2 & 0 & \dots & 0 \\ 
0 & B & I_2 & \dots & 0 \\ 
\vdots & \vdots & \ddots & \ddots & \vdots \\ 
0 & 0 & \dots & B & I_2 \\ 
0 & 0 & \dots & 0 & B
\end{pmatrix},
\end{equation}
where $B$ has the form of $(\ref{eq34})$.
We claim that $A_k=B$ for all $k$, $1\leq k\leq M.$
Indeed, if there exist $k_0$ which different to $B$, then by \cite[Proposition 3.7]{ABS},
the operator presented by the matrix
$$
C=\begin{pmatrix}
B & I_2 \\ 
0 & B
\end{pmatrix}
$$
is super-recurrent in $\mathbb{R}^4$.
The same argument used in the complex case lead us to a contradiction.

Hence, if $T$ is super-recurrent, then $A$ is similar to a matrix which has the form of $(\ref{eq24})$.\\
\noindent $(4)\Rightarrow (3)$ : It is not difficult that each matrix of that form is uniformly super-rigid.
\end{proof}

\end{document}